\documentclass{amsart}[12pt]
\usepackage[hidelinks]{hyperref}
\usepackage{verbatim}
\usepackage{yfonts} %
\usepackage{amssymb} %
\usepackage{amsthm}
\usepackage{array}
\usepackage{booktabs}%
\usepackage{hhline}%
\usepackage{xy} %
\usepackage{epsfig}%
\usepackage{color}%
\usepackage{upgreek}
\usepackage[english]{babel}
\usepackage{epigraph}%
\usepackage{fancybox}%
\usepackage{shadow}
\usepackage{afterpage}
\usepackage{mathrsfs}
\usepackage{enumitem}
\usepackage{tabularx}
\usepackage{subcaption}
\usepackage{graphicx}
\usepackage{type1cm}
\usepackage{eso-pic}
\usepackage{color}
\usepackage{upgreek}
\usepackage[foot]{amsaddr}
\usepackage{bigints}
\usepackage{thmtools}
\usepackage{thm-restate}
\usepackage{cleveref}

\newtheorem{claim}{Claim}

\newtheorem{lemma}{Lemma}

\theoremstyle{definition}
\newtheorem{definition}{Definition}
\newtheorem{remark}{Remark}

\theoremstyle{plain}

\newtheorem*{kt}{Generalized Koiso's Theorem}

\newcommand{\otoprule}{\midrule[\heavyrulewidth]}

\newcommand{\vt}{\vspace{.1cm}}
\newcommand{\vtt}{\vspace{.2cm}}
\newcommand{\R}{\mathbb{R} }
\newcommand{\q}{\mathbb{Q} }

\newcommand{\h}{\mathbb{H}}

\newcommand{\s}{\mathbb{S}}

\renewcommand{\rho}{\varrho}
\renewcommand{\theta}{\varTheta}
\renewcommand{\Theta}{\varTheta}
\renewcommand{\Sigma}{\varSigma}
\renewcommand{\Omega}{\varOmega}
\renewcommand{\Lambda}{\varLambda}
\renewcommand{\tau}{\uptau}
\usepackage{amsmath}

\newcommand{\overbar}[1]{\mkern 1.5mu\overline{\mkern-1.5mu#1\mkern-1.5mu}\mkern 1.5mu}

\newcommand{\qr}{\q_\epsilon^n\times\R}

\makeatletter
\newcommand{\tpitchfork}{%
  \vbox{
    \baselineskip\z@skip
    \lineskip-.52ex
    \lineskiplimit\maxdimen
    \m@th
    \ialign{##\crcr\hidewidth\smash{$-$}\hidewidth\crcr$\pitchfork$\crcr}
  }%
}
\makeatother

\begin{document}

\title[Stability and  Isoperimetry  of CMC Spheres]
{On Stability and Isoperimetry of Constant Mean Curvature
Spheres  of $\mathbb H^n\times\mathbb R$ and $\mathbb S^n\times\mathbb R.$}
\author{R. F. de Lima, M. F. Elbert \and B. Nelli}
\address[A1]{Depto. de Matem\'atica - Univ. Federal do Rio Grande do Norte - RN, Brazil.}
\email{ronaldo.freire@ufrn.br}
\address[A2]{Instituto de Matemática - Universidade Federal do Rio de Janeiro - RJ, Brazil.}
\email{fernanda@im.ufrj.br}
\address[A3]{Dipartimento di Ingegneria e Scienze dell'Informazione e Matematica, Universitá dell’Aquila, Italy.}
\email{barbara.nelli@univaq.it}

\maketitle

\begin{abstract}
  We approach the one-parameter family of rotational constant mean curvature (CMC)  spheres
  of $\mathbb H^n\times\mathbb R$  and $\mathbb S^n\times\mathbb R$ focusing on their
  stability and isoperimetry properties.
  We prove that  all rotational CMC spheres of $\mathbb H^n\times\mathbb R$ are stable,
  and that the ones in $\mathbb S^n\times\mathbb R$ with sufficiently small (resp.~large)
  mean curvature are unstable (resp.~stable). We also show that there exists
  a one-parameter family of stable CMC rotational spheres in $\mathbb S^n\times\mathbb R$
  which are not isoperimetric (i.e., they do not bound isoperimetric regions).
  We  establish the uniqueness of the regions enclosed by the rotational
  CMC spheres of $\mathbb H^n\times\mathbb R$ as solutions
  to the isoperimetric problem, filling in a gap in the original
  proof given by Hsiang and Hsiang. We establish, as well,
  a sharp upper bound for the volume of the spherical regions of $\mathbb S^n\times\mathbb R$
  which are unique solutions to the isoperimetric problem.
  In essence, all these results come from the fact that the rotational CMC spheres of
  $\mathbb H^n\times\mathbb R$, and those of $\mathbb S^n\times\mathbb R$
  with sufficiently large mean curvature, are nested.

  \vspace{.15cm}
  \vspace{.1cm}
\end{abstract}

\section{Introduction}

Compact hypersurfaces of constant mean curvature (CMC)
in  Riemannian manifolds are known to be critical points of the area
functional $\mathcal A$
for  volume-preserving variations.
A CMC hypersurface is then called
\emph{stable} if it minimizes $\mathcal A$ up to second order for
any such variation. In simply connected space forms, as established
in~\cite{barbosaetal}, geodesic spheres are the only compact stable
hypersurfaces.

Souam~\cite{souam} considered the rotational CMC spheres of
$\h^2\times\R$,  previously obtained by
Hsiang and Hsiang~\cite{hsiang},
proving that they are all stable, and that those with
(nonnormalized) constant mean curvature $H\ge\sqrt2$
are unique with respect to this property.
He considered the rotational CMC spheres  of $\s^2\times\R$ as well, constructed by
Pedrosa in~\cite{pedrosa}, and showed that such a sphere is stable if and only if its
mean curvature is greater than, or equal to, a constant which is approximately $0.36$.
Furthermore, any compact stable CMC (non-minimal) surface of $\s^2\times\R$ is necessarily a rotational sphere.

Stability is a concept closed related to isoperimetry. Recall that
an embedded compact hypersurface $\Sigma$ of a Riemannian manifold $\mathbb M^{n+1}$ is called
\emph{isoperimetric} if the compact region $\Omega\subset\mathbb M^{n+1}$ having $\Sigma$ as boundary is
the solution to the isoperimetric problem for prescribed volume $\mathcal V:={\rm Vol}(\Omega).$
In this case, we  call $\Omega$ an \emph{isoperimetric region}.
This means that, among all regions of $\mathbb M^{n+1}$ of volume $\mathcal V,$ $\Omega$ is the one
whose boundary has least area, where by \emph{volume} and \emph{area} we mean
the $(n+1)$-dimensional and $n$-dimensional Riemannian measures, respectively.
A standard result, then, establishes that
any $C^2$-smooth isoperimetric hypersurface is necessarily CMC and stable.
However, there exist stable embedded  CMC hypersurfaces which are not isoperimetric, such as
the  geodesic spheres of large radius in the real projective space (cf.~\cite[Appendix]{barbosaetal}).

In~\cite{hsiang}, Hsiang and Hsiang showed that,
for any $H>n-1\ge 1,$ there exists
an embedded rotational  sphere $\Sigma_H$
of constant (nonnormalized) mean curvature $H$ in $\h^n\times\R$
which is necessarily isoperimetric.
In addition,  by means of Alexandrov reflections, they established that
any compact embedded hypersurface in $\h^n\times\R$
of constant mean curvature $H>0$ is  congruent to
a rotational sphere $\Sigma_H.$ In particular, $H$ must be larger than $n-1.$
In fact, by the tangency principle, the mean curvature
of any compact immersed CMC hypersurface of $\h^n\times\R$ must be larger than
$n-1$, since there exists an entire vertical graph in $\h^n\times\R$ of constant mean curvature
$H$ for any $H\in(0,n-1]$ (cf.~\cite{delima-manfio-santos,elbert-earp}).

It is also shown in~\cite{hsiang} that the  spherical region $\Omega_H$
with boundary $\Sigma_H$ is the only solution to the isoperimetric problem for
prescribed volume $\mathcal V={\rm Vol}(\Omega_H).$
However, the proof the authors provide is incorrect (cf.~Remark~\ref{rem-flawed}
in Section~\ref{sec-proofs}).

Similar results regarding the isoperimetric problem in
$\s^n\times\R$ were obtained by Pedrosa~\cite{pedrosa}.
He proved that there exists an embedded  rotational  sphere
$\Sigma_H$ of constant mean curvature $H$ in $\s^n\times\R$ for any given $H>0,$
and that a solution to the isoperimetric problem is either the spherical region
$\Omega_H$ bounded by a rotational sphere
$\Sigma_H$ or a cylindrical section $\s^n\times[a,b].$
In addition, if $H$ is sufficiently large,
$\Omega_H$ is the only solution
to the isoperimetric problem in $\s^n\times\R$
for prescribed volume $\mathcal V={\rm Vol}(\Omega_H).$
In his work, Pedrosa also obtained a sharp lower bound
(approximately 0.66) for the mean curvature of an isoperimetric
sphere of $\s^2\times\R.$ Notice that, in what concerns stability, the lower bound of $H$
obtained by Souam is smaller (approximately 0.36). Therefore, any rotational
CMC sphere of $\s^2\times\R$ having mean curvature between these
two lower bounds is stable and non-isoperimetric.

In this paper, inspired by the aforementioned works,
we establish results on stability and isoperimetry of
CMC spheres of $\h^n\times\R$ and $\s^n\times\R.$
Our approach brings to light that, in this context,
the fundamental properties of the  CMC spheres are
intimately related with the \emph{nesting property}, which can be described as follows.

First, denote by $\q_\epsilon^n$ the simply
connected $n(\ge\hspace{-.1cm}2)$-space form of constant
sectional curvature $\epsilon\in\{-1,1\}$  (i.e., $\h^n=\q_{-1}^n$ and $\s^n=\q_{1}^n$).
Then, set $C_\epsilon:=n-1$ for $\epsilon=-1$ and $C_\epsilon:=0$ for $\epsilon=1,$ and
write $\mathscr F_\epsilon$ for the family of all noncongruent
rotational  spheres $\Sigma_H$ in $\qr$ of constant
mean curvature $H>C_\epsilon.$ Assume that all
spheres of $\mathscr F_\epsilon$ have the same rotational axis
$\{o\}\times\R\subset\qr$
and the same point $o\in\q_\epsilon^n\times\{0\}$
of minimal height. Finally, as above, denote by $\Omega_H$ the compact spherical region
of $\qr$ enclosed by $\Sigma_H.$
In this setting, for a given $C\ge C_\epsilon,$ we say that the
subfamily $\mathscr F_C:=\{\Sigma_H\,;\, H>C\}\subset\mathscr F_\epsilon$ is \emph{nested} if
$\Omega_H\subset\Omega_{H^*}$ whenever $H>H^*>C.$

Our first result, as stated below, describes the behavior of the
CMC spheres of the family $\mathscr F_\epsilon$ regarding the nesting property.
Its proof will be based on the fact that any sphere $\Sigma_H\in\mathscr F_\epsilon$ has
a horizontal hyperplane $P_H$ of symmetry, and that a subfamily
$\mathscr F_C\subset\mathscr F_\epsilon$ is nested if and only if the height of
$P_H$ is a decreasing function of $H$ for $H>C$  (see Fig.~\ref{fig-nestednotnested}).
\begin{figure}[ht]
\includegraphics[scale=1.2]{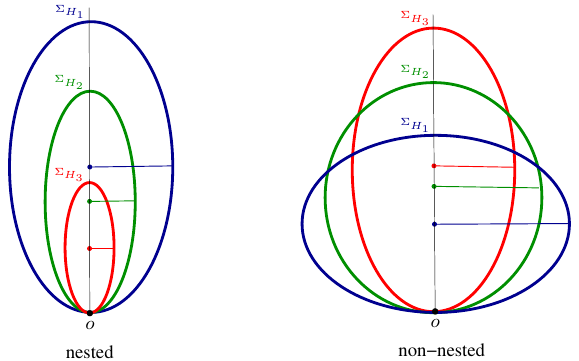}
\caption{\small Nested and non-nested families of rotational CMC spheres $\Sigma_{H_i},$ $i=1,2,3$,
of mean curvatures $H_1<H_2<H_3$. As the mean curvature increases from $H_1$ to $H_3,$
the height of the horizontal hyperplanes of symmetry (represented by the horizontal lines)
decreases in the nested case, and increases in the non-nested case.}
\label{fig-nestednotnested}
\end{figure}

\begin{restatable}{thm}{nested} \label{th-nestedspheres}
Let $\mathscr F_\epsilon:=\{\Sigma_H \,;\,  H>C_\epsilon\}$ be the family of
rotational CMC spheres of \,$\qr.$ Then, the following hold:
\begin{itemize}[parsep=1ex]
\item[\rm{\bf A)}] For $\epsilon=-1$, the whole family $\mathscr F_\epsilon$ is nested.
\item[\rm{\bf B)}] For $\epsilon=1,$ there exists $C_{\min}>C_\epsilon:=0$ such that $\mathscr F_{C}$ is
nested if and only if $C\ge C_{\min}.$ Furthermore, as $H$ decreases to $0$,
the height function of $\Sigma_H\in\mathscr F_\epsilon$
strictly decreases to $0,$ and $\Sigma_H$ converges uniformly to (a double copy of)
the totally geodesic horizontal hyperplane $\s^n\times\{0\}.$
\end{itemize}
\end{restatable}

Theorem~\ref{th-nestedspheres}, together with~\cite[Theorems 1 and 2]{hsiang}
and~\cite[Theorem 2]{pedrosa}, yields the following uniqueness result, which,
in the case of $\s^n\times\R,$ was previously obtained by
Pedrosa with a different proof (cf.~\cite[Theorem 5]{pedrosa}).

\begin{restatable}{thm}{uniquenessone}\label{th-uniquenessPI}
Given $\mathcal V_0>0$  (resp.~a sufficiently small $\mathcal V_0>0$),
there exists a spherical region $\Omega_{H_0}$ in \,$\h^n\times\R$ (resp.~$\s^n\times\R$)
which is, up to ambient isometries, the only solution to the isoperimetric problem
for prescribed volume $\mathcal V_0.$
\end{restatable}

Next, we apply  Theorem \ref{th-nestedspheres}
to establish  the following extension of the result by Souam we mentioned before;
namely~\cite[Theorem 2.2]{souam},
which is set in $\q_\epsilon^2\times\R.$
We remark that, in what concerns stability and isoperimetry of
rotational CMC spheres of $\s^n\times\R,$ our result
ensures the existence of one-parameter families of
all types: stable, unstable, isoperimetric, and stable non-isoperimetric.

\begin{restatable}{thm}{stability}  \label{th-stabilityH-spheres}
Let $\mathscr F_\epsilon:=\{\Sigma_H \,;\,  H>C_\epsilon\}$ be the family of
rotational CMC spheres of \,$\qr.$ Then, the following hold:
\begin{enumerate}[parsep=1ex]
\item[\rm{\bf A)}] For $\epsilon=-1,$ any sphere $\Sigma_H$ of $\mathscr F_\epsilon$ is stable.
\item[\rm{\bf B)}] For $\epsilon=1,$ there exist positive numbers {$H_0<H_1$} such that:
\begin{itemize}[parsep=1ex]
\item[\rm i)] $\Sigma_H$ is stable if $H\ge H_0.$
\item[\rm ii)]$\Sigma_H$ is unstable if $H<H_0$ is sufficiently close to $H_0.$
\item[\rm iii)] $\Sigma_H$ is stable and non-isoperimetric if $H_0\le H<H_1$.
\item[\rm iv)]$\Sigma_H$ is isoperimetric if $H\ge H_1$.
\end{itemize}
\end{enumerate}
\end{restatable}

The result in Theorem~\ref{th-stabilityH-spheres}-{\bf A} follows directly from the fact that
the rotational spheres of $\h^n\times\R$ are all isoperimetric. Nevertheless,
we considered relevant to give it an independent proof, as did Souam for the case $n=2.$

Regarding Theorem~\ref{th-stabilityH-spheres}-{\bf B}, we point out that
the stability of a rotational sphere $\Sigma_H$ of \,$\s^n\times\R$ depends on the
growth behavior of its area  $\mathcal A=\mathcal A(H)$
considered as a function of its mean curvature $H$ (this is also true for
the rotational spheres of $\h^n\times\R,$ but we will not use it here).

As we shall see, stable spheres are those for which $\mathcal A'(H)\le 0$,
whereas the unstable ones are those for which $\mathcal A'(H)>0$.
For $n=2,$  $\mathcal A$ can be expressed in terms
of elementary functions of $H$ (cf.~\cite[Section 4]{pedrosa}).
Using this expression, Souam concluded that $\mathcal A$ has only one
critical point $H_0$ (necessarily a maximum), which is precisely the
sharp lower bound  for the mean curvature of stable spheres of $\s^2\times\R$
he obtained.
For $n>2,$ instead,  $\mathcal A$ is defined  by means of an integral of elliptic type which
is not explicitly computable, making  a complete determination of the growth behavior of $\mathcal A$
nearly impossible. Yet, we believe that $\mathcal A$ has only one critical point, just as in the case $n=2$
(see Figure~\ref{fig-areagraph} in Section~\ref{sec-proofs}, which depicts
an illustrative  graph of $\mathcal A$ as a function
of the ``radiuses'' of the rotational spheres).

Our final result constitutes a refinement of Theorem~\ref{th-uniquenessPI}
in the case $\epsilon=1.$ It establishes a sharp upper bound for the volume of
the spherical regions $\Omega_H$ which are unique solutions to the isoperimetric
problem.

\begin{restatable}{thm}{uniqueness}\label{thm-uniquenessPI}
Given an integer $n\ge 2,$ there exists a positive constant $\overbar{\mathcal V}=\overbar{\mathcal V}(n)$
which, regarding
the isoperimetric problem for any prescribed volume $\mathcal V_0>0$ in $\s^n\times\R,$
has the following properties:
\begin{itemize}[parsep=1ex]
\item[\rm i)] If \,$\mathcal V_0<\overbar{\mathcal V}$, {any solution is necessarily a
spherical region $\Omega_H\subset\s^n\times\R$}.

\item[\rm ii)] If \,$\mathcal V_0>\overbar{\mathcal V}$, there exists a cylindrical section $\Omega_{[a,b]}:=\s^n\times[a,b]$ which is a solution.

\item[\rm iii)] If \,$\mathcal V_0=\overbar{\mathcal V}$, there exist a spherical region
and a cylindrical section which are both solutions.
\end{itemize}
Consequently, if a spherical region $\Omega_H\subset\s^n\times\R$ is a unique solution
to the isoperimetric problem for prescribed volume $\mathcal V_0:={\rm Vol}(\Omega_H),$
then \,${\rm Vol}(\Omega_H)<\overbar{\mathcal V}.$ {Moreover, the constant
$\overbar{\mathcal V}$ is sharp in the following sense:
for any $\delta\in(0,\overbar{\mathcal V})$,
there exists $\mathcal V_0\in(\overbar{\mathcal V}-\delta,\overbar{\mathcal V})$, and a  spherical region
of volume $\mathcal V_0$ which is  the unique solution.}
\end{restatable}

For $n=2,$ the lower bound for the mean curvature of isoperimetric spheres obtained by
Pedrosa gives that the upper bound  $\overbar{\mathcal V}$ of their volumes is approximately
$16.66$. Moreover, any spherical region of volume less than $\overbar{\mathcal V}$ is unique
as solution to the isoperimetric problem. When $n>2$, we cannot ensure this
uniqueness property since it is not clear, although likely,
that the volume of the spherical regions $\Omega_H$ is not constant when $H$
varies in an open interval.

The paper is organized as follows. In Section~\ref{sec-stability},
we briefly discuss stability of CMC hypersurfaces in Riemannian manifolds, and introduce
a result by Koiso which establishes necessary and sufficient
conditions (on the Jacobi operator) for stability of CMC hypersurfaces of complete Riemannian manifolds.
In Section~\ref{sec-lemmas}, we prove three  technical lemmas to be applied in
the proofs of Theorems~\ref{th-nestedspheres}--\ref{thm-uniquenessPI}, which will be presented
in  Section~\ref{sec-proofs}.

\section{Stability of CMC hypersurfaces} \label{sec-stability}

Let $\Sigma$ be a compact oriented hypersurface of a Riemannian manifold $\mathbb M^{n+1}$,  $n\ge 2,$ and recall that
the nonnormalized mean curvature function $H$ of $\Sigma$
is the trace of its shape operator, that is,
the sum of its principal curvatures. Suppose that $\Sigma=\Phi_0(M),$ where
$\Phi_0\colon M\to\mathbb M^{n+1}$ is an immersion of a compact
$n$-dimensional manifold $M$  into $\mathbb M^{n+1}.$ Given a smooth variation $\Phi(t,.)$
of $\Phi_0=\Phi(0,.),$ $t\in (-\delta,\delta),$
its \emph{area} and \emph{balance of volume} functionals are defined as (see~\cite{barbosa-colares,barbosaetal})
\[
\mathcal A(t):=\int_{M}dM_t \quad\text{and}\quad V(t):=\int_{[0,t]\times M}\Phi^*d(\mathbb M^{n+1}),
\]
where $dM_t$  is the area element of $M$ in the metric induced by the immersion $\Phi(t,.)$, and
$\Phi^*d(\mathbb M^{n+1})$ is the pull-back by $\Phi$ of the volume element of $\mathbb M^{n+1}$.
In this setting, one defines the \emph{variational field} $\xi$ of $\Phi(t,.)$ as
\[
\xi:=\frac{\partial\Phi}{\partial t}(t, . )|_{t=0}.
\]
It is well known that (cf. \cite[Lemma 2.1]{barbosaetal})
\begin{equation}  \label{eq-aprimeandvprime}
\mathcal A'(0)=-\int_{\Sigma}Hf \quad\text{and}\quad V'(0)=\int_{\Sigma}f,
\end{equation}
where $f:=\langle\xi, N\rangle$ and  $N$ is the unit normal field of $\Sigma.$
From these identities, one sees that
compact CMC hypersurfaces are critical points of the functional area for
volume-preserving variations (i.e., those satisfying $V'(0)=0$), as we mentioned before.
A compact CMC hypersurface is then called \emph{stable} if $\mathcal A''(0)\ge 0$ for such variations.

The stability of a CMC hypersurface $\Sigma$ is equivalent to the
nonnegativeness of the quadratic form $Q$  defined by
\[
Q(u,u)=-\int_\Sigma uLu
\]
for functions $u\in C^{\infty}(\Sigma)$ satisfying $\int_\Sigma u=0$ (cf. \cite[Proposition 2.7]{barbosaetal}).
Here, $L$ is the (self-adjoint) \emph{Jacobi operator} on $C^{\infty}(\Sigma),$ which is defined as
\[
Lu:=\Delta u+(\|\sigma\|^2+{\rm Ric}(N))u,
\]
being $\Delta u$  the Laplacian of $u$ on $\Sigma,$  $\sigma$  the second fundamental form of $\Sigma,$
and ${\rm Ric}(N)$  the Ricci curvature of $\mathbb M^{n+1}$ evaluated on the unit normal field
$N$ of $\Sigma.$
We add that the Jacobi operator $L$ of the CMC hypersurface $\Sigma$ satisfies the following  identity:
\begin{equation} \label{eq-fundamental}
\frac{\partial H}{\partial t}(t, . )|_{t=0}=Lf.
\end{equation}

In \cite{koiso}, Koiso established a fundamental result
on stability of CMC surfaces with boundary in Euclidean space $\R^3.$ As pointed out by Souam in \cite{souam2},
her result is easily extended to compact oriented CMC hypersurfaces  with empty boundary in any Riemannian manifold,
giving the following result (see also  \cite{damasceno-elbert},
where the authors established a broad extension of Koiso's Theorem).

\begin{kt}
Let $\Sigma$ be a smooth compact
immersed oriented CMC hypersurface in a complete $(n+1)$-dimensional Riemannian manifold \,$\mathbb M^{n+1},$
and let $\lambda_1$ and $\lambda_2$
be the first and second eigenvalues of the Jacobi operator $L$ on
$\Sigma,$ respectively. Under these assumptions, the following hold.
\begin{enumerate}
\item[\rm{\bf A)}] Assume $\lambda_1<0=\lambda_2.$ Then, one has:
\vt
\begin{itemize}[parsep=1ex]
\item[\rm i)]  If there exists an eigenfunction $u$ associated to $\lambda_2=0$ such that $\int_{\Sigma}u\ne 0$, then
$\Sigma$ is unstable.
\item[\rm ii)] If $\int_{\Sigma}u=0$ for any eigenfunction $u$ associated to $\lambda_2=0,$
then  there exists a unique smooth function $v\in(\ker L)^\perp$ (orthogonal with respect to the $L^2$ norm)
which satisfies $Lv=1$ and determines the stability condition of $\Sigma;$ namely,
$\Sigma$ is stable if and only if $\int_{\Sigma}v\ge 0.$
\end{itemize}

\vt

\item[\rm{\bf B)}]  If $\lambda_2<0,$ then $\Sigma$ is unstable.
\end{enumerate}
\end{kt}

\section{Three Technical Lemmas} \label{sec-lemmas}
The proof of Theorem \ref{th-nestedspheres} will be based on the
fundamental properties of the solutions
of the first order linear ODE:
\begin{equation}  \label{eq-odevarphi}
y'=-(n-1)\cot_\epsilon(s) y+1,
\end{equation}
where $n>1$ is an integer,
$\cot_\epsilon:=\cos_\epsilon/\sin_\epsilon,$
and  $\cos_\epsilon, \, \sin_\epsilon$ are as in Table \ref{table-trigfunctions}.
\begin{table}[thb]%
\centering %
\begin{tabular}{ccc}
\toprule %
{{\small\rm Function}}     &      $\epsilon=1$   & $\epsilon=-1$ \\\otoprule %
$\cos_\epsilon $        &      $\cos $       & $\cosh $     \\\midrule
$\sin_\epsilon $        &      $\sin $       & $\sinh $     \\\bottomrule
\end{tabular}
\vtt
\caption{Definition of $\cos_\epsilon$ and $\sin_\epsilon.$}
\label{table-trigfunctions}
\end{table}

Setting
\begin{equation} \label{eq-R}
\mathcal R_\epsilon:=\left\{
\begin{array}{ccl}
+\infty & \text{if} & \epsilon=-1,\\
\pi  & \text{if} & \epsilon=1,
\end{array}
\right.
\end{equation}
it is easily checked that the function $\varphi_\epsilon\colon[0,\mathcal R_\epsilon)\to\R$ defined by
\begin{equation}\label{eq-varphiepsilon}
\varphi_\epsilon(s):=
\left\{
\begin{array}{ll}
\frac{1}{\sin_\epsilon^{n-1}(s)}\int_{0}^{s}{\sin_\epsilon^{n-1}(u)}du & \text{if} \,\, s\ne 0,\\ [1.5ex]
                              0                                        & \text{if} \,\, s= 0
\end{array}
\right.
\end{equation}
is a smooth solution to \eqref{eq-odevarphi} (see Fig.~\ref{fig-varphiplots}).

\begin{lemma} \label{lem-varphi}
The function $\varphi_\epsilon\colon[0,\mathcal R_\epsilon)\rightarrow\R$ has the following
properties:
\begin{itemize}[parsep=1ex]
  \item [\rm i)] $\varphi_\epsilon'(s)>0$ for all $s\in (0,\mathcal R_\epsilon).$
  \item[\rm ii)] $\varphi_\epsilon'(0)=1/n.$
  \item[\rm iii)] $\epsilon\varphi_\epsilon''(s)>0 \,\, \forall s\in(0,\mathcal R_\epsilon).$
\end{itemize}
\end{lemma}
\begin{proof}
To prove (i), we first observe that
$\varphi_\epsilon$ is increasing near $0,$ since it vanishes at $0$
and is  positive in $(0,\mathcal R_\epsilon).$
Assume, by contradiction,  that  $\varphi_\epsilon$ is not
increasing in the whole interval $(0,\mathcal R_\epsilon)$. In this case,
$\varphi_\epsilon$ has a first critical point $s_1$  which is necessarily a local maximum.

Since $\varphi_\epsilon$ is a solution to \eqref{eq-odevarphi}, one has
\begin{equation}  \label{eq-odevarphiproof}
\varphi_\epsilon'(s)=-(n-1)\cot_\epsilon(s)\varphi_\epsilon(s)+1, \,\,\, s\in(0,\mathcal R_\epsilon).
\end{equation}
So,  $\varphi_\epsilon''(s_1)=(n-1)\csc_\epsilon^2(s_1)\varphi_\epsilon(s_1)>0,$
which gives that $s_1$ is  a local minimum --- a contradiction. This proves (i).
\begin{figure}[htbp]
\includegraphics[scale=.25]{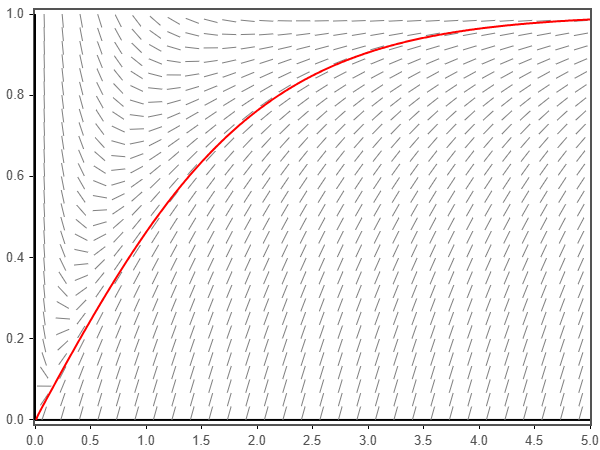}
\hspace{1cm}
\includegraphics[scale=.25]{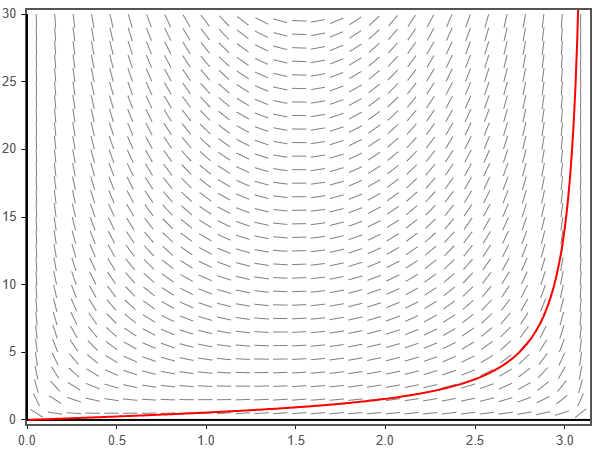}
\caption{Graphs of the solutions $\varphi_\epsilon$ to the ODE \eqref{eq-odevarphi} for $\epsilon=-1$ (left) and $\epsilon=1$ (right).}
\label{fig-varphiplots}
\end{figure}

Regarding (ii), we have from  \eqref{eq-varphiepsilon} and \eqref{eq-odevarphiproof} that
\begin{equation} \label{eq-varphiprime}
\varphi_\epsilon'(s)=-(n-1)\frac{\cos_\epsilon(s)\mathcal I_{\epsilon}(s)}{\sin_\epsilon^{n}(s)}+1,
\end{equation}
where $\mathcal I_{\epsilon}(s):=\int_{0}^{s}{\sin_\epsilon^{n-1}(u)}du.$ By the L'Hôpital rule,
\begin{eqnarray*}
\lim_{s\rightarrow 0}\frac{\cos_\epsilon(s)\mathcal I_{\epsilon}(s)}{\sin_\epsilon^n(s)} &=&
\frac{1}{n}\lim_{s\rightarrow 0}\left(\frac{-\epsilon \mathcal I_{\epsilon}(s)}{\sin_\epsilon^{n-2}(s)}+\cos_\epsilon(s)\right)\\
   &=& \frac{1}{n}\lim_{s\rightarrow 0}\left(-\epsilon\varphi_\epsilon(s)\sin_\epsilon(s)+\cos_\epsilon(s)\right)=1/n,
\end{eqnarray*}
which, together with \eqref{eq-varphiprime}, yields (ii).

To prove  (iii),  we first observe that,
for $n=2,$ one has from equality \eqref{eq-varphiepsilon} that
$\varphi_\epsilon(s)=\epsilon(1-\cos_\epsilon(s))/\sin_\epsilon(s).$
In this case, it can be easily checked that $\epsilon\varphi_\epsilon''>0$ in $(0,\mathcal R_\epsilon).$
So, we can assume $n>2.$ Under this hypothesis, let us show that
\begin{equation} \label{eq-omega}
\Gamma_\epsilon(s):=\sin_\epsilon^n(s)-(n-1)\cos_\epsilon(s)\mathcal I_{\epsilon}(s)>0 \,\,\forall s\in(0,\mathcal R_\epsilon).
\end{equation}

Since $\Gamma_\epsilon(0)=0,$ it suffices to prove that $\Gamma_\epsilon'>0$ in $(0,\mathcal R_\epsilon).$
From a direct computation, we get
\[
\Gamma_\epsilon'(s)=\sin_\epsilon(s)(\epsilon(n-1)\mathcal I_{\epsilon}(s)+\cos_\epsilon(s)\sin_\epsilon^{n-2}(s)).
\]
Setting $\mu_\epsilon(s):=\epsilon(n-1)\mathcal I_{\epsilon}(s)+\cos_\epsilon(s)\sin_\epsilon^{n-2}(s),$
one has $\Gamma_{\epsilon}'(s)=\sin_\epsilon(s)\mu_\epsilon(s)$. Besides,
$\mu_\epsilon(0)=0$ and
\[
\mu_\epsilon'(s)=(n-2)\sin_\epsilon^{n-3}(s)(\cos_\epsilon^2(s)+\epsilon\sin_\epsilon^2(s))=
(n-2)\sin_\epsilon^{n-3}(s)>0 \,\,\, \forall s\in(0,\mathcal R_\epsilon).
\]
Therefore, $\mu_\epsilon,$ and so $\Gamma_\epsilon',$ is positive  in $(0,\mathcal R_\epsilon).$
This finishes the proof of \eqref{eq-omega}.

Now, we have from \eqref{eq-varphiprime} that
\begin{equation} \label{eq-varphitwoprime}
\varphi_\epsilon''(s)=(n-1)\frac{\Lambda_{\epsilon}(s)}{\sin_\epsilon^{n+1}(s)},
\end{equation}
where $\Lambda_{\epsilon}$ is the function
\[
\Lambda_{\epsilon}(s):=\epsilon\sin_\epsilon^2(s)\mathcal I_{\epsilon}(s)-\cos_\epsilon(s)(\sin_\epsilon^n(s)-n\cos_\epsilon(s)\mathcal I_{\epsilon}(s)).
\]
Again, we have $\Lambda_{\epsilon}(0)=0.$ In addition, it can be easily checked that
\[
\Lambda_{\epsilon}'(s)=2\epsilon\sin_\epsilon(s)\Gamma_\epsilon(s),
\]
which yields $\epsilon\Lambda_{\epsilon}'>0$ in $(0,\mathcal R_\epsilon).$
Hence, $\epsilon\Lambda_{\epsilon}(s)$ is positive in $(0,\mathcal R_\epsilon).$
From this and \eqref{eq-varphitwoprime},
we finally have $\epsilon\varphi_\epsilon''>0$ in $(0,\mathcal R_\epsilon),$ as we wished to prove.
\end{proof}


\begin{lemma} \label{lem-psi}
For $\epsilon=1,$ the function $\psi\colon(0,\pi)\to\R$ defined by
\[
\psi(s):=\frac{\varphi_\epsilon'(s)}{\varphi_\epsilon(s)}
\]
has a unique extreme point $s_*\in (0,\pi)$ which is a
global minimum. Furthermore, $\psi$ is {strictly} decreasing in
$(0,s_*)$ and {strictly} increasing in $(s_*,\pi).$
\end{lemma}

\begin{proof}
  Firstly, let us observe that there is no interval $I\subset (0,\pi)$ on which $\psi$
  is constant. Indeed, assuming otherwise, there is a constant $\psi_0>0$ such that, for all $s\in I,$
  $\psi(s)=\psi_0,$  that is, $\varphi_\epsilon'(s)=\psi_0\varphi_\epsilon(s).$
  Hence, there is a constant $C>0$ such that $\varphi_\epsilon(s)=Ce^{\psi_0s}$, $s\in I.$
  However, this  function is clearly distinct from the one defined in~\eqref{eq-varphiepsilon}.
  Thus, $\psi$ is constant in no interval, which implies that $\psi$ is increasing (resp.~decreasing)
  in an interval $I\subset(0,\pi)$ if and only if $\psi$ is strictly increasing (resp.~strictly decreasing) in $I.$

  By items (i) and (ii) of Lemma \ref{lem-varphi}, we have that
  \begin{equation} \label{eq-psi01}
  \lim_{s\to 0}\psi(s)=+\infty.
  \end{equation}
  Also,  from equality~\eqref{eq-odevarphiproof},
  \begin{equation} \label{eq-psi02}
  \lim_{s\to\pi}\psi(s)=\lim_{s\to\pi}\frac{\varphi_\epsilon'(s)}{\varphi_\epsilon(s)}=
  \lim_{s\to\pi}\left[-(n-1)\cot(s)+\frac{1}{\varphi_\epsilon(s)}\right]=+\infty.
  \end{equation}

  It follows from~\eqref{eq-psi01} and~\eqref{eq-psi02} that $\psi$ is bounded from below.
  In particular, $\psi$ has a  global minimum $s_*\in(0,\mathcal R_\epsilon).$

  Now, set $a(s)=-(n-1)\cot s$. Then, $\psi=a+1/\varphi_\epsilon,$ which yields
  \[
  \psi'=a'-\frac{\psi}{\varphi_\epsilon} \quad\text{and}\quad \psi''=a''-\frac{\psi'}{\varphi_\epsilon}+\frac{\psi^2}{\varphi_\epsilon}\cdot
  \]
  Hence, if $s_c\in (0,\pi)$ is a critical point of $\psi,$ we have
  \[
  a'(s_c)=\frac{\psi(s_c)}{\varphi_\epsilon(s_c)} \quad\text{and}\quad \psi''(s_c)=a''(s_c)+\frac{\psi^2(s_c)}{\varphi_\epsilon(s_c)}\cdot
  \]
  Observing that $a''(s)=-2\cot(s)a'(s)$ for all $s\in(0,\pi),$ we can write
  \[
  \psi''(s_c)=-2\cot(s_c)a'(s_c)+\frac{\psi^2(s_c)}{\varphi_\epsilon(s_c)}=
  -2\cot(s_c)\frac{\psi(s_c)}{\varphi_\epsilon(s_c)}+\frac{\psi^2(s_c)}{\varphi_\epsilon(s_c)},
  \]
  which implies that
  \begin{equation} \label{eq-psi''}
  \psi''(s_c)=\frac{\psi(s_c)}{\varphi_\epsilon(s_c)}(\psi(s_c)-2\cot(s_c)).
  \end{equation}

Suppose, by contradiction, that $\psi$ is not {strictly}
increasing in $(s_*,\pi).$ Under this hypothesis,
$\psi$ has a critical point $s_c\in(s_*,\pi),$ which is not a local minimum,
such that $\psi$ is {strictly} increasing in $(s_*,s_c).$ Hence, if we define the function
\[
\tau(s):=\psi(s)-2\cot(s), \,\,\, s\in(0,\pi),
\]
we have  from \eqref{eq-psi''} that $\tau(s_c)\le 0\le\tau(s_*).$ However, $\tau'(s)=\psi'(s)+2\csc^2(s)$
is positive in $(s_*,s_c),$ which is clearly a contradiction.
Therefore, $\psi$ is {strictly} increasing in the interval $(s_*,s_c).$

Assume now that $\psi$ is not {strictly} decreasing in $(0,s_*).$ Then, since $s_*$ is an absolute minimum for
$\psi$, and $\psi(s)\to+\infty$ as $s\to 0,$ there exist $s_1<s_2$ in $(0,s_*)$ such that
$s_1$ and $s_2$ are a minimum and a maximum for $\psi,$ respectively, and $\psi$ is {strictly} increasing
in $(s_1,s_2).$ From this point, we can derive a contradiction by reasoning just as in the preceding paragraph.
Therefore, $\psi$ is {strictly} decreasing in $(0,s_*).$ This concludes the proof.
\end{proof}


\begin{lemma} \label{lem-Fepsilon}
Given $s_0\in (0,\mathcal R_\epsilon),$ let $F_{(\epsilon, s_0)}\colon [0,s_0]\rightarrow\R$ be the function defined by
\begin{equation} \label{eq-Fepsilon}
F_{(\epsilon, s_0)}(s):=\varphi_\epsilon(s_0)\varphi_\epsilon'(s)-\varphi_\epsilon(s)\varphi_\epsilon'(s_0).
\end{equation}
Then,  the following assertions hold (see Fig. \ref{fig-Fgraphs}):
\begin{enumerate}
\item[\rm{\bf A)}] For $\epsilon=-1,$ $F_{(\epsilon,s_0)}$ is positive in $[0,s_0)$ for all $s_0>0.$

\vt

\item[\rm{\bf B)}]  For $\epsilon=1,$ let $s_*$ be as in Lemma {\rm \ref{lem-psi}}.
Then, one has:
\vt
\begin{itemize}[parsep=1ex]
\item[\rm i)]  $F_{(\epsilon,s_0)}$ is positive in $[0,s_0)$ if and only if $s_0\in(0,s_{*}].$
\item[\rm ii)] For any $s_0\in (s_*,\pi),$ there exists $\alpha(s_0)\in(0,s_*)$ such that
$F_{(\epsilon,s_0)}$ is nonnegative in $[0,\alpha(s_0)]$ and negative in $(\alpha(s_0),s_0).$
Moreover, $\alpha(s_0)$ decreases to $0$ as $s_0\to\pi.$
\end{itemize}
\end{enumerate}
\end{lemma}

\begin{proof}
We shall consider first the case $\epsilon=-1.$
From Lemma \ref{lem-varphi}-(ii), we have  that
$F_{(\epsilon,s_0)}(0)=\varphi_\epsilon(s_0)\varphi_\epsilon'(0)=\varphi_\epsilon(s_0)/n>0.$ Since
$F_{(\epsilon,s_0)}(s_0)=0,$  it suffices to prove
that, for any $s_0>0,$  $F_{(\epsilon,s_0)}$ is
strictly decreasing in $[0,s_0).$

For the derivative of $F_{(\epsilon,s_0)},$ we have
\[
F_{(\epsilon,s_0)}'(s)=\varphi_\epsilon(s_0)\varphi_\epsilon''(s)-\varphi_\epsilon'(s)\varphi_\epsilon'(s_0).
\]
Thus,  from Lemma \ref{lem-varphi}-(iii), $F_{(\epsilon,s_0)}'<0$ in $(0,s_0],$
which implies that $F_{(\epsilon,s_0)}$ is strictly decreasing. This proves {\bf A}.

\begin{figure}[htbp]
\includegraphics[scale=.30]{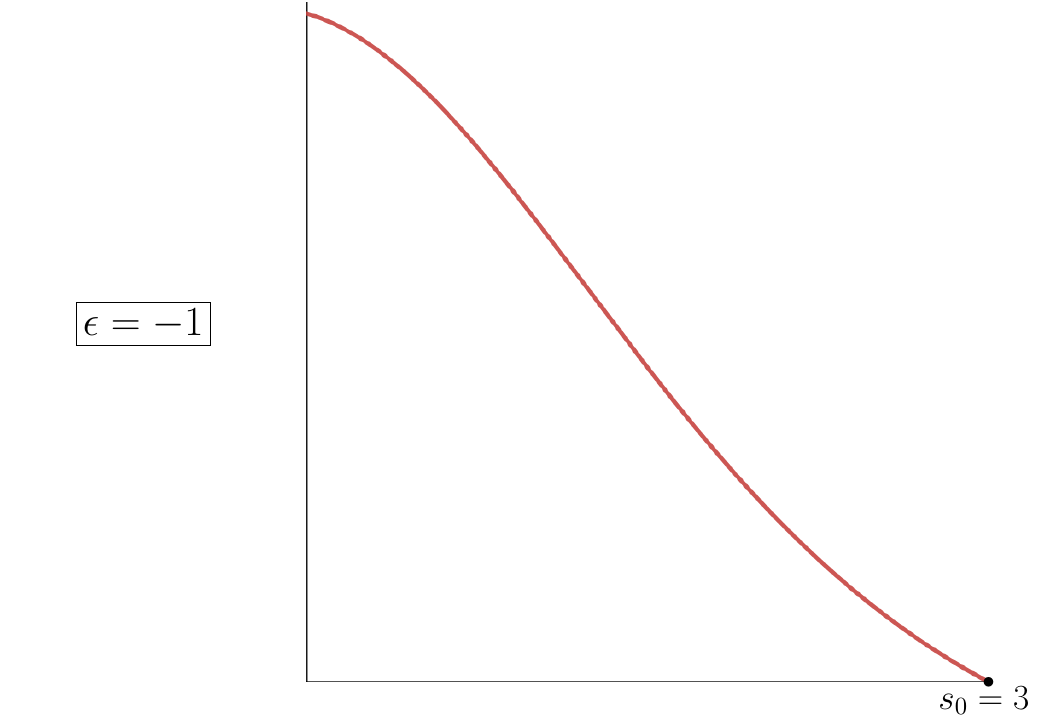}
\hspace{.08cm}
\includegraphics[scale=.30]{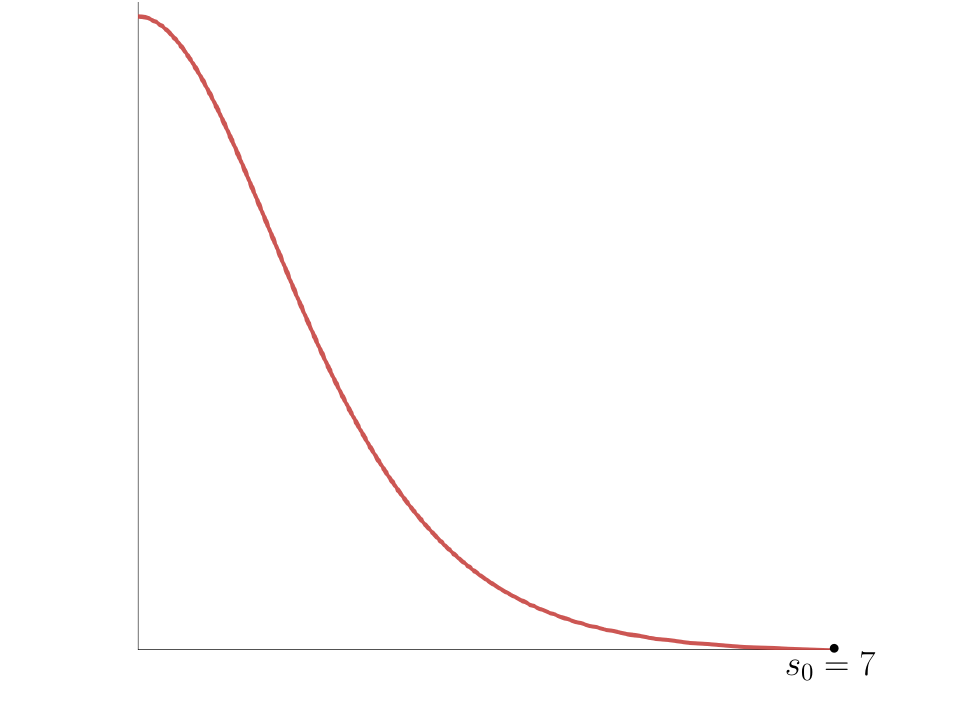}\\
\vspace{.2cm}
\includegraphics[scale=.32]{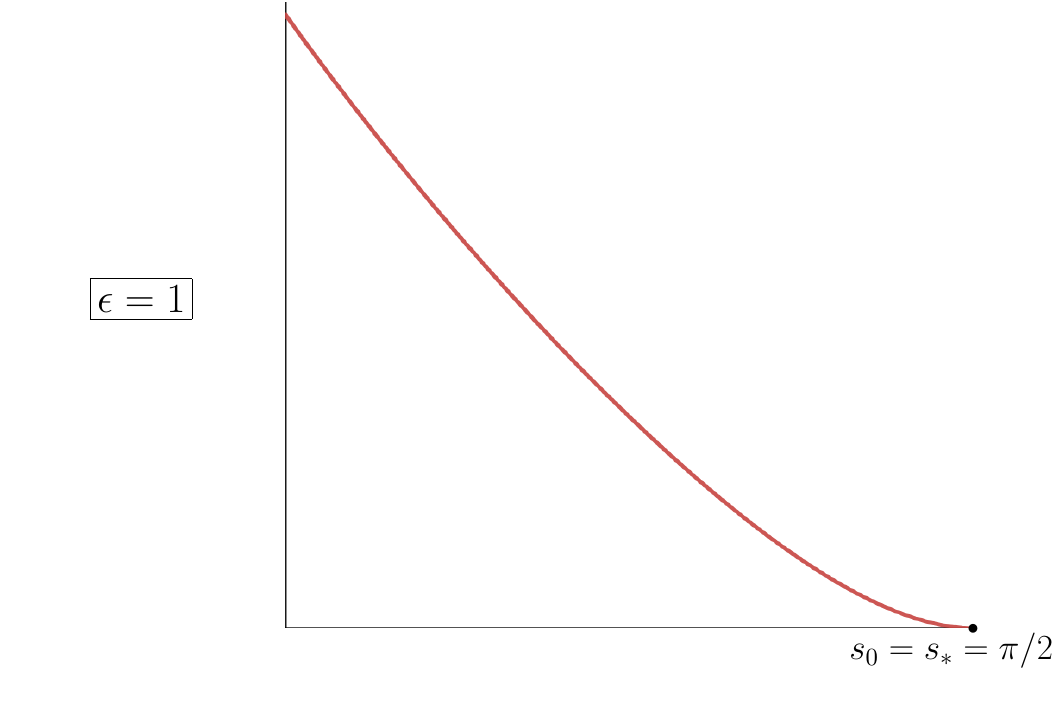}
\hspace{.08cm}
\includegraphics[scale=.27]{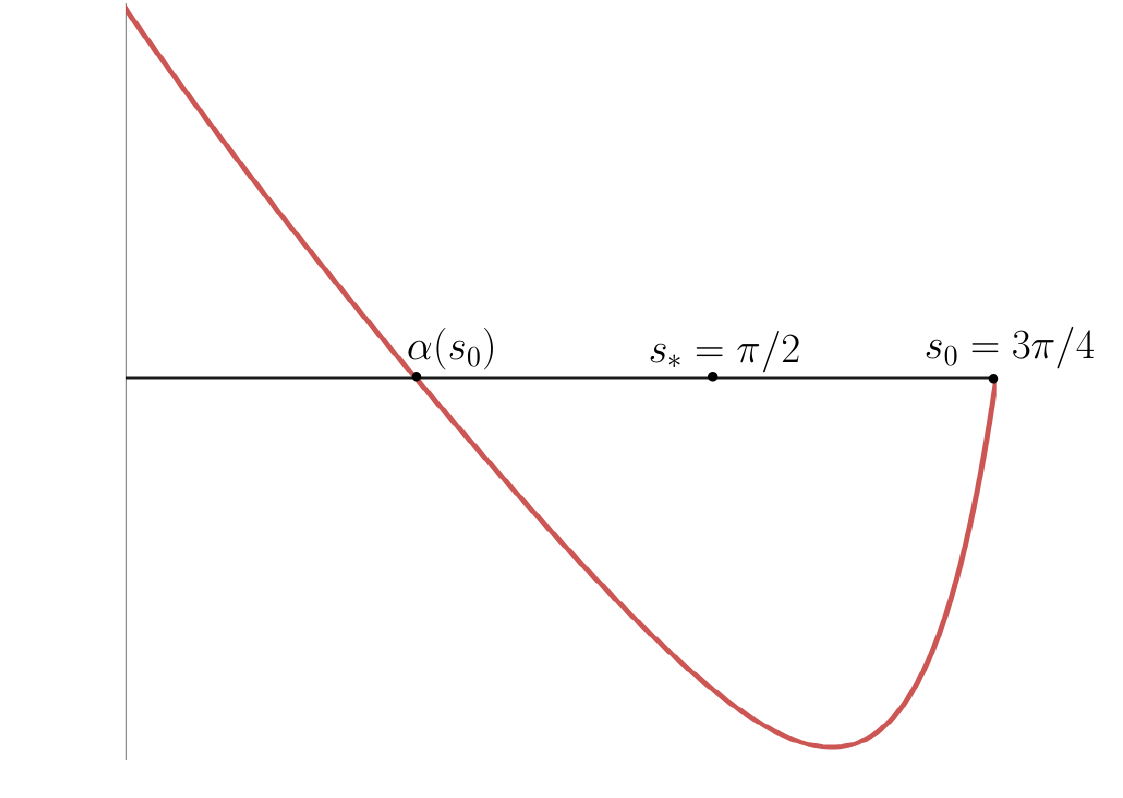}\\
\caption{Graphs of $F_{(\epsilon,s_0)}\colon [0,s_0]\to\R$ for $n=2$ (in which case $s_*=\pi/2$).
As indicated, the ones on the top correspond to the values
$(\epsilon, s_0)=(-1,3)$ (left) and $(\epsilon, s_0)=(-1,7)$ (right), whereas the ones on the bottom correspond to the values
$(\epsilon, s_0)=(1,\pi/2)$ (left) and $(\epsilon, s_0)=(1,3\pi/4)$ (right).}
\label{fig-Fgraphs}
\end{figure}

Assume now that $\epsilon=1.$  Consider the function $\psi$
introduced in Lemma \ref{lem-psi} and choose $s_0\in(0,\pi).$ It is easily seen that
\begin{equation} \label{eq-iff}
{F_{(\epsilon,s_0)}(s)> 0 \,\,\Leftrightarrow\,\, \psi(s)> \psi(s_0) \quad\text{and}\quad F_{(\epsilon,s_0)}(s)=0 \,\,\Leftrightarrow\,\, \psi(s)=\psi(s_0).}
\end{equation}

Assertion  {\bf B}-(i) follows
directly from \eqref{eq-iff} and Lemma~\ref{lem-psi}.
Regarding {\bf B}-(ii), notice  that the equivalences \eqref{eq-iff} and Lemma~\ref{lem-psi}
also imply that $F_{(\epsilon,s_0)}$ is negative in $[s_*,s_0)$ for all $s_0\in(s_*,\pi).$
Thus, since $F_{(\epsilon,s_0)}(0)>0,$ there exists a unique  $\alpha=\alpha(s_0)>0$ with the following properties:
\begin{equation} \label{eq-alpha0}
F_{(\epsilon,s_0)}(\alpha(s_0))=0, \,\,\,\, F_{(\epsilon,s_0)}'(\alpha(s_0))<0, \,\,\,\, \text{and} \,\,\,\,\,
F_{(\epsilon,s_0)}|_{[0,\alpha(s_0)]}\ge 0.
\end{equation}
Under these conditions, $F_{(\epsilon,s_0)}$ is necessarily negative in $(\alpha(s_0),s_*].$ Otherwise, there
would exist $s_1<s_2\in(\alpha(s_0),s_*)$ such that $F_{(\epsilon,s_0)}(s_1)<0\le F_{(\epsilon,s_0)}(s_2).$ But,
by Lemma \ref{lem-psi}, $\psi(s_1)>\psi(s_2).$ This, together with \eqref{eq-iff}, would then give
\[
\psi(s_0)>\psi(s_1)>\psi(s_2)\ge\psi(s_0),
\]
which is  a contradiction. Hence, $F_{(\epsilon,s_0)}$ is negative in
$(\alpha(s_0),s_0)$ for all $s_0\in(s_*,\pi).$

To prove the final statement in {\bf B}-(ii), let us first observe that, by Lemma \ref{lem-psi},
$\psi'(s_0) {\ge} 0$ whenever $s_0\in(s_*,\pi).$ In such case, for $s\in (0,s_0),$ we have
\begin{equation} \label{eq-partialFs0}
\frac{\partial}{\partial s_0}\left(\frac{F_{(\epsilon,s_0)}(s)}{\varphi_\epsilon(s_0)}\right)=
\frac{\partial}{\partial s_0}(\varphi'(s)-\psi(s_0)\varphi(s))=-\psi'(s_0)\varphi(s){\le} 0.
\end{equation}
This means that, for any $\bar s_0>s_0$ and $s\in (0,s_0)\subset (0,\bar s_0),$ the inequality
\[
\frac{F_{(\epsilon,\bar s_0)}(s)}{\varphi_\epsilon(\bar s_0)}{\le}\frac{F_{(\epsilon,s_0)}(s)}{\varphi_\epsilon(s_0)}
\]
holds. Applying it to $s=\alpha(s_0)$ yields $F_{(\epsilon,\bar s_0)}(\alpha(s_0))\le 0,$ so that
$\alpha(\bar s_0)\le \alpha(s_0).$ Therefore, the function $s_0\in (s_*,\pi)\mapsto\alpha(s_0)$ is
positive, continuous, and  {decreasing}, which implies that
\[
\lim_{s_0\to\pi}\alpha(s_0)=\mathcal L:=\inf\{\alpha(s_0)\,;\, s_0\in (s_*,\pi)\}.
\]

To finish the proof, assume by contradiction that $\mathcal L>0.$ Under this assumption, we have
from \eqref{eq-alpha0} and the definition of infimum that
$F_{(\epsilon,s_0)}(\mathcal L)\ge 0$ for all $s_0\in(s_*,\pi).$ However, considering \eqref{eq-psi02}, we also have
\[
\lim_{s_0\to\pi}\frac{F_{(\epsilon,s_0)}(\mathcal L)}{\varphi_{\epsilon}(s_0)}=
\lim_{s_0\to\pi}(\varphi_{\epsilon}'(\mathcal L)-\psi(s_0)\varphi_{\epsilon}(\mathcal L))=-\infty,
\]
so that $F_{(\epsilon,s_0)}(\mathcal L)<0$  for any  sufficiently large $s_0\in(s_*,\pi)$, which is
a contradiction. Therefore, $\mathcal L=0,$ as we wished to prove.
\end{proof}


\section{Proofs of Theorems \ref{th-nestedspheres}--\ref{thm-uniquenessPI} } \label{sec-proofs}

In this section, we shall approach the rotational CMC spheres
$\Sigma_H\in\mathscr F_\epsilon$ of $\qr$ as in \cite{delima-manfio-santos}, where the more general
case of hypersurfaces of constant higher order mean curvature was considered
(see also \cite{elbert-earp}).

Recall from the introduction that, for each $H>C_\epsilon$, and for a fixed point
$o\in\mathbb Q_\epsilon^n\times\{0\}$, there exists a unique
rotational sphere $\Sigma_H\in\mathscr F_\epsilon$ of constant mean curvature $H$ and axis
$\{o\}\times\R,$
where $C_\epsilon=0$ for $\epsilon=1$ and $C_\epsilon=n-1$ for $\epsilon=-1.$
In~\cite{delima-manfio-santos}, it is proved that, for each $H>C_\epsilon,$
there exists $s_0=s_0(H)\in(0,\mathcal R_\epsilon)$ such that
$\Sigma_H$ is given by the union of the closure of
a radial graph $\Sigma_H^-$ over
the open geodesic ball $B(o,s_0)\subset\q_\epsilon^n\times\{0\}$
with its reflection about a totally geodesic horizontal hyperplane
$P_H\subset\qr,$ so that $P_H$ is a hyperplane
of symmetry for $\Sigma_H$ (Fig.~\ref{fig-sphere}).

\begin{definition}
We shall call $s_0=s_0(H)\in(0,\mathcal R_\epsilon)$ the \emph{radius} of the
sphere $\Sigma_H\in\mathscr F_\epsilon$.
\end{definition}

\begin{figure}[htbp]
\includegraphics{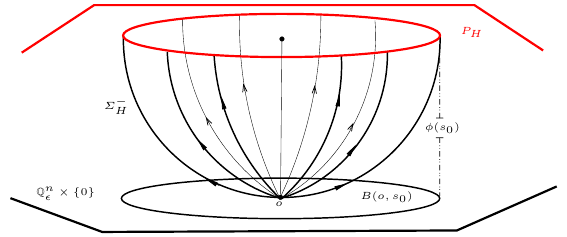}
\caption{A rotational CMC hemisphere in $\qr.$}
\label{fig-sphere}
\end{figure}

As discussed in~\cite[Section 4]{delima-manfio-santos}, the  height  function of  $\Sigma_{H}^-$ is
\begin{equation}  \label{eq-phi}
\phi(s)=\int_{0}^{s}\frac{\rho(u)}{\sqrt{1-\rho^2(u)}}du, \,\,\,\, s\in [0,s_0),
\end{equation}
where $\rho\colon[0,s_0)\to[0,1)$ is defined as
\begin{equation}\label{eq-tausn}
\rho(s):=H\varphi_\epsilon(s),  \,\,\, s\in[0,s_0),
\end{equation}
being $\varphi_\epsilon$ the function introduced
in~\eqref{eq-varphiepsilon}. It is then established that $\rho$ extends
smoothly to $s_0$ and satisfies $\rho(s_0)=1.$ Thus,
since $\varphi_\epsilon$ is clearly injective on $(0,\mathcal R_\epsilon)$,
we have from \eqref{eq-tausn} that the mean curvature $H$ of the sphere
$\Sigma_H\in\mathscr F_\epsilon$ is an injective function of its radius $s_0$ and
satisfies
\begin{equation} \label{eq-H(s0)}
H(s_0)=\frac{1}{\varphi_\epsilon(s_0)}, \,\,\,  s_0\in(0,\mathcal R_\epsilon).
\end{equation}

Summarizing, we have that the correspondence
\begin{equation}  \label{eq-correspondence}
H=H(s_0)\in(C_\epsilon,+\infty)\,\leftrightarrow\, s_0=s_0(H)\in(0,\mathcal R_\epsilon)
\end{equation}
is bijective. Moreover, from \eqref{eq-H(s0)} and Lemma \ref{lem-varphi}-(i), $H=H(s_0)$ is
strictly decreasing, which implies that $s_0=s_0(H)$ is strictly decreasing as well.
It is also clear that, in this setting,  the horizontal plane of symmetry
of $\Sigma_H$ is $P_H:=\q_\epsilon^n\times\{\phi(s_0)\}.$

For $\epsilon=1,$ it is easily checked from the
definition of $\varphi_\epsilon$ in \eqref{eq-varphiepsilon} that
\begin{equation} \label{eq-limitvarphi}
\lim_{s_0\to\pi}\varphi_\epsilon(s_0)=+\infty.
\end{equation}
Besides, $\varphi_\epsilon(0)=0.$ These facts and \eqref{eq-H(s0)} then yield
\begin{equation} \label{eq-limitHzeroandpi}
\lim_{s_0\to 0}H(s_0)=+\infty \quad\text{and}\quad \lim_{s_0\to\pi}H(s_0)=0.
\end{equation}

Finally, we recall that
a subfamily
$\mathscr F_C:=\{\Sigma_H\,;\, H>C\}\subset\mathscr F_\epsilon$
is  said to be \emph{nested} if
$\Omega_H\subset\Omega_{H^*}$ whenever $H>H^*>C,$ where
$\Omega_H$ is the region of $\qr$ bounded by the sphere $\Sigma_H\in\mathscr F_\epsilon.$

Now, we are in position to prove Theorem~\ref{th-nestedspheres},
which we restate for the reader's convenience.

\nested*

\begin{proof}
Keeping the above notation, assume $H>H^*$
and denote the height and $\rho$ functions of
$\Sigma_{H^*}^-$  by $\phi^*$ and $\rho^*,$ respectively.
Writing $s_0^*:=s_0(H^*),$ we have $s_0<s_0^*,$
since $s_0$ is a decreasing function of $H$.
Then, considering \eqref{eq-tausn}, we get
\[
\rho(s)>\rho^*(s) \,\forall s\in (0, s_0),
\]
which, together with \eqref{eq-phi}, implies that
\begin{equation} \label{eq-inequalityphi}
\phi(s)>\phi^*(s) \,\,\,\forall s\in(0, s_0).
\end{equation}

It is clear  from \eqref{eq-inequalityphi} that  $\Omega_H\subset\Omega_{H^*}$ if and only if
$\phi(s_0^*)>\phi(s_0)$ (see Fig.~\ref{fig-nestedhemispheres}).
In other words, a subfamily $\mathscr F_C:=\{\Sigma_H\,;\, H>C\}$
of $\mathscr F_\epsilon$ is nested if and only if
\[
s_0=s_0(H)\mapsto \phi(s_0), \,\,\, H>C,
\]
is an increasing function of $s_0$. This suggests us
to introduce the radius $s_0$ as a variable for $\rho$. More precisely, we consider
the correspondence \eqref{eq-correspondence}, as well as equality \eqref{eq-H(s0)},
and  rewrite~\eqref{eq-tausn} as
\begin{equation}  \label{eq-rho&H}
\rho(s,s_0):=H(s_0)\varphi_\epsilon(s)=\frac{\varphi_\epsilon(s)}{\varphi_\epsilon(s_0)}\cdot
\end{equation}
In this setting, we have from~\eqref{eq-phi} that
\begin{equation} \label{eq-phi(s0)}
\phi(s_0)=\int_{0}^{s_0}\frac{\rho(s,s_0)}{\sqrt{1-\rho^2(s,s_0)}}ds.
\end{equation}

\begin{figure}[htbp]
\includegraphics{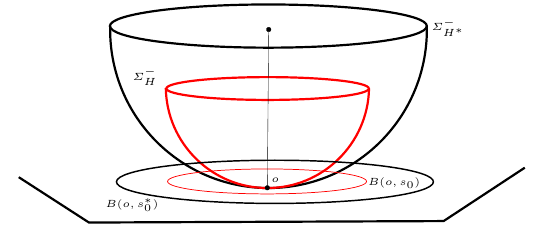}
\caption{Nested CMC hemispheres in $\qr.$}
\label{fig-nestedhemispheres}
\end{figure}

Calling $J(s,s_0)$ the integrand of the integral in~\eqref{eq-phi(s0)},
an application of the Leibniz rule (see~\cite[Lemma 5.1]{pedrosa}) gives
\[
\frac{\partial\phi}{\partial s_0}(s_0)=\int_{0}^{s_0}\left(\frac{\partial J}{\partial s_0}(s,s_0)
+\frac{\partial J}{\partial s}(s,s_0)\right)ds.
\]
Then, from a straightforward calculation, we get
\begin{equation} \label{eq-partialphi}
\frac{\partial\phi}{\partial s_0}(s_0) =
\int_{0}^{s_0}{\frac{1}{\sqrt{(1-\rho^2(s,s_0))^3}}}\left(\frac{\partial\rho}{\partial s}(s,s_0)+
\frac{\partial\rho}{\partial s_0}(s,s_0)\right)ds.
\end{equation}

Considering~\eqref{eq-rho&H}, we have
\begin{eqnarray}
\frac{\partial\rho}{\partial s}(s,s_0)+\frac{\partial\rho}{\partial s_0}(s,s_0)
&=&  \frac{\varphi_\epsilon'(s)}{\varphi_\epsilon(s_0)}-\frac{\varphi_\epsilon'(s_0)}{\varphi_\epsilon^2(s_0)}\varphi_\epsilon(s)\nonumber\\
&=&  \frac{1}{\varphi_\epsilon^2(s_0)}\left(\varphi_\epsilon'(s)\varphi_\epsilon(s_0)-\varphi_\epsilon'(s_0)\varphi_\epsilon(s)\right)\nonumber\\
&=&  \frac{1}{\varphi_\epsilon^2(s_0)}F_{(\epsilon,s_0)}(s), \nonumber
\end{eqnarray}
where $F_{(\epsilon,s_0)}$ is the function defined in \eqref{eq-Fepsilon}. This last equality and
\eqref{eq-partialphi} then yield
\begin{equation} \label{eq-partialphi007}
\frac{\partial\phi}{\partial s_0}(s_0) =
\frac{1}{\varphi_\epsilon^2(s_0)}\int_{0}^{s_0}{\frac{F_{(\epsilon,s_0)}(s)}{\sqrt{(1-\rho^2(s,s_0))^3}}}ds.
\end{equation}

It follows from \eqref{eq-partialphi007} and Lemma \ref{lem-Fepsilon}-{\bf A}
that, for $\epsilon=-1,$  $\partial\phi/\partial s_0>0$ at any $s_0>0.$ So, in this case,
the whole family $\mathscr F_\epsilon$ is nested. This proves statement {\bf A} of the theorem.

For $\epsilon=1,$ Lemma \ref{lem-Fepsilon}-{\bf B} gives that $F_{(\epsilon,s_0)}$ is positive
in $[0,s_0)$ whenever $s_0\in (0,s_*],$ where $s_*$ is as in Lemma \ref{lem-psi}. Thus,
the subfamily $\mathscr F_{H(s_*)}$ is nested, which implies that
$C_{\min}:=\inf\{C\ge 0\,;\, \mathscr F_C \,\, \text{is nested}\}$
is well defined and satisfies $C_{\min}\le H(s_*).$

Let us show now that $C_{\min}>0.$ With this purpose, we set
\[
\Lambda(s,s_0):={\frac{1}{\sqrt{(1-\rho^2(s,s_0))^3}}}\ge 1.
\]
Since $\rho(s,s_0)=\varphi_\epsilon(s)/\varphi_\epsilon(s_0),$  we conclude from \eqref{eq-limitvarphi} that,
for any fixed $s\in (0,s_0),$  $\Lambda(s,s_0)$ decreases to $1$ as $s_0\to\pi.$

Now, for a given $s_0\in (s_*,\pi),$ we
have from \eqref{eq-partialphi007} that (notation as in Lemma \ref{lem-Fepsilon})
\begin{eqnarray}
\frac{\partial\phi}{\partial s_0}(s_0) & = & \frac{1}{\varphi_\epsilon(s_0)}\left[
\int_{0}^{\alpha(s_0)}\Lambda(s,s_0)\frac{F_{(\epsilon,s_0)}(s)}{\varphi_\epsilon(s_0)}ds+
\int_{\alpha(s_0)}^{s_0}\Lambda(s,s_0)\frac{F_{(\epsilon,s_0)}(s)}{\varphi_\epsilon(s_0)}ds\right]\nonumber\\
                                       & \le & \frac{1}{\varphi_\epsilon(s_0)}\left[
\int_{0}^{\alpha(s_0)}\Lambda(s,s_0)\frac{F_{(\epsilon,s_0)}(s)}{\varphi_\epsilon(s_0)}ds+
\int_{\alpha(s_0)}^{s_0}\frac{F_{(\epsilon,s_0)}(s)}{\varphi_\epsilon(s_0)}ds\right],\label{eq-part}
\end{eqnarray}
where, in \eqref{eq-part}, we used  that $F_{(\epsilon,s_0)}$ is negative in
$(\alpha(s_0),s_0)$ and that $\Lambda\ge 1.$

Recall that, for any fixed $s\in (0,s_0),$   the function
$s_0\in(s_*,\pi)\mapsto F_{(\epsilon,s_0)}(s)/\varphi_\epsilon(s_0)$ is
decreasing (cf.~\eqref{eq-partialFs0}). Besides, Lemma~\ref{lem-Fepsilon} gives that
$\alpha(s_0)$ decreases to $0$ as $s_0$ increases to $\pi.$ These facts clearly imply that
the first integral in~\eqref{eq-part} decreases to $0$ as $s_0$ increases to $\pi.$

On the other hand, from \eqref{eq-psi02}, and again from
the fact that $F_{(\epsilon,s_0)}$ is negative in
$(\alpha(s_0),s_0),$ one has, for each fixed $s\in(0,s_0),$
\[
\lim_{s_0\to\pi}\frac{F_{(\epsilon,s_0)}(s)}{\varphi_\epsilon(s_0)}=
\lim_{s_0\to\pi}(\varphi'(s)-\psi(s_0)\varphi(s))=-\infty,
\]
so that the second integral in~\eqref{eq-part} decreases
to $-\infty$ as $s_0$ increases to $\pi.$ Therefore, at any $s_0>s_*$  sufficiently close
to $\pi,$ $\partial\phi/\partial s_0$ is negative at $s_0,$ which implies that the family
$\mathscr F_\epsilon$ is not entirely nested, that is,  $C_{\min}>0,$ as we wished to prove.

It follows from the discussion in the preceding paragraphs that,
in an open interval $(\pi-\delta,\pi)\subset (s_*,\pi),$
the function $\phi=\phi(s_0)$
decreases to its infimum as $s_0$ increases to $\pi$.
To see that $\inf\phi=0,$ we first observe that, for $\epsilon=1,$  the principal curvatures of
a sphere $\Sigma_{H(s_0)}\in\mathscr F_\epsilon$ are the functions
(cf. \cite[Section 4]{delima-manfio-santos}):
\begin{itemize}[parsep=1ex]
\item $k_i(s,s_0)=\cot(s)\rho(s,s_0), \,\,\, i=1,\dots, n-1$,
\item $k_n(s,s_0)=\rho'(s,s_0).$
\end{itemize}
But,  for all fixed $s\in (0,s_0)$, we have from
\eqref{eq-limitvarphi} and \eqref{eq-rho&H} that
\[
\lim_{s_0\to\pi}\rho(s,s_0)=\lim_{s_0\to\pi}\rho'(s,s_0)=0.
\]
Therefore,
\begin{equation} \label{eq-limitkis}
\lim_{s_0\to\pi}k_i(s,s_0)=0, \,\,\, i=1,\dots, n.
\end{equation}

Since, for $s_0$ near $\pi,$ the height function of $\Sigma_{H(s_0)}$ is a strictly decreasing function of
$s_0$,
the equalities \eqref{eq-limitkis} imply that, as $s_0\to\pi,$  the lower
hemispheres $\Sigma_{H(s_0)}^-$ converge uniformly to
the totally geodesic hyperplane $\s^n\times\{0\}$ (Fig.~\ref{fig-tapioca}).
In particular, $\inf\phi=0.$ This proves statement  {\bf B} and concludes the proof of the theorem.
\end{proof}

\begin{figure}[hb]
\includegraphics[scale=1]{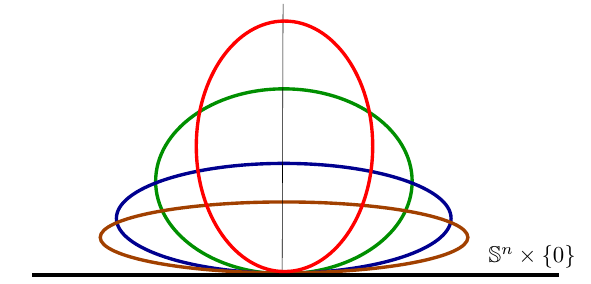}
\caption{\small The rotational CMC spheres $\Sigma_H$ of $\s^n\times\R$ converge to a double copy of $\s^n\times\{0\}$
as their mean curvatures decrease to $0.$}
\label{fig-tapioca}
\end{figure}

\begin{remark} \label{rem-flawed}
The case $\epsilon=-1$ of Theorem \ref{th-nestedspheres} was considered in \cite{hsiang} as
part of the proof of the uniqueness of the rotational CMC spheres
as isoperimetric hypersurfaces of  $\h^n\times\R.$ However, the proof
the authors provided is incorrect. More precisely,
on page 48 of \cite{hsiang}, the  equality of the first line
of equations in (11) should be
\[
\frac{\partial^2y}{\partial h\partial u}=
-\left(\frac{\partial u}{\partial y}\right)^{-2}\left(\frac{\partial^2u}{\partial h\partial y}+
\frac{\partial^2u}{\partial y^2}\frac{\partial y}{\partial h}\right),
\]
but the second summand of the second factor
is missing there.
Therefore, the reasoning that follows is flawed. Furthermore, if that proof in \cite{hsiang} were correct, one
could easily adapt it to show that the entire family $\mathscr F_\epsilon$
of rotational CMC spheres of $\s^n\times\R$ is nested, which is
not true, by Theorem \ref{th-nestedspheres}. As a matter of fact, this inconsistency suggested us
the incorrectness  of the proof in \cite{hsiang}.
\end{remark}

Next, we restate and prove
Theorems~\ref{th-uniquenessPI}--\ref{thm-uniquenessPI}.

\uniquenessone*
\begin{proof}
Since $\qr$ is homogeneous, there exists a solution to the isoperimetric problem for
any preassigned volume $\mathcal V_0>0$ (cf.~\cite{almgren}).
For $\epsilon=-1,$~\cite[Theorems 1 and 2]{hsiang}
ensure that any such solution is necessarily congruent to a spherical region $\Omega_{H_0}$.
However, by Theorem~\ref{th-nestedspheres},  $\mathscr F_\epsilon$ is nested for $\epsilon=-1.$
Therefore, there exists only one $H_0>n-1$ such that $\Omega_{H_0}$ has volume $\mathcal V_0.$

Assume now that $\epsilon=1.$ By \cite[Theorem 2]{pedrosa}, an isoperimetric region of $\s^n\times\R$ is either
a cylindrical section $\Omega_{[a,b]}:=\s^n\times [a,b]$ or a spherical region $\Omega_H$.
It is easily checked that
\begin{equation} \label{eq-areavolume}
{\rm Area}\,(\partial\Omega_{[a,b]})=2\omega_n \quad\text{and}\quad  {\rm Vol}\,(\Omega_{[a,b]})=(b-a)\omega_n,
\end{equation}
where $\omega_n$ is the area of \,$\s^n.$ In particular, given a spherical region $\Omega_H$,
any cylindrical section $\Omega_{[a,b]}$ such that
$b-a={\rm Vol}\,(\Omega_H)/\omega_n$ has the same volume as $\Omega_H$. This, together with the first
equality in \eqref{eq-areavolume}, proves the following

\begin{claim} \label{claim-isoperimetric}
Given $H>0$, consider the spherical region $\Omega_H\subset\s^n\times\R$, the cylindrical section $\Omega_{[a,b]}$ of
volume $\mathcal V:={\rm Vol}(\Omega_H)$,
and the isoperimetric problem for prescribed volume $\mathcal V$.
Then, the following assertions hold:
\begin{itemize}[parsep=1ex]
\item $\Omega_H$ is the only solution  if and only if ${\rm Area}(\Sigma_H)<2\omega_n$.
\item $\Omega_{[a,b]}$ is the only solution if and only if ${\rm Area}\,(\Sigma_H)>2\omega_n$.
\item $\Omega_H$ and $\Omega_{[a,b]}$ are  both solutions if and only if ${\rm Area}\,(\Sigma_H)=2\omega_n$.
\end{itemize}
\end{claim}

Considering the area of $\Sigma_H$ and the volume of $\Omega_H$ as functions of the radius $s_0=s_0(H)$
of $\Sigma_H,$ we have from~\cite[Proposition 4.1]{pedrosa} that
\begin{equation}\label{eq-areaandvolume}
  \begin{aligned}
    {\rm Area}\,(\Sigma_H) &= 2\omega_{n-1}\int_{0}^{s_0}\frac{\sin^{n-1}(s)}{\sqrt{1-\rho^2(s,s_0)}}ds,\\[1ex]
    {\rm Vol}\,(\Omega_H) &=2\omega_{n-1}\int_{0}^{s_0}\frac{\mathcal I(s)\rho(s,s_0)}{\sqrt{1-\rho^2(s,s_0)}}ds,
  \end{aligned}
\end{equation}
where $\rho$ is as in~\eqref{eq-rho&H} and $\mathcal I(s):=\int_{0}^{s}\sin^{n-1}(s).$
Clearly, both integrals in these equalities converge to $0$ as $s_0\to 0.$ From this and
\eqref{eq-limitHzeroandpi}, we get
\begin{equation} \label{eq-limitareavolume}
\lim_{H\to+\infty}{\rm Area}\,(\Sigma_H)=\lim_{H\to+\infty}{\rm Vol}\,(\Omega_H)=0.
\end{equation}

It follows from Claim~\ref{claim-isoperimetric} and equalities~\eqref{eq-limitareavolume} that there exist $\delta, C_\delta>0$
having the following property: for all $\mathcal V_0\in (0,\delta),$
there exists $H_0>C_\delta$ such that the spherical region
$\Omega_{H_0}$ is a solution to the isoperimetric problem
for prescribed volume $\mathcal V_0.$
We can assume $\delta$ sufficiently small, so that
$C_\delta\ge C_{\min},$ where $C_{\min}$ is as in Theorem~\ref{th-nestedspheres}. In this way,
the solution $\Omega_{H_0}$ is unique, for the subfamily
$\mathscr F_{C_{\min}}\subset\mathscr F_\epsilon$ is nested, by
Theorem~\ref{th-nestedspheres}. This concludes the proof.
\end{proof}

\stability*
\begin{proof}
With the purpose of applying Koiso's Theorem  (see Section~\ref{sec-stability}),
we first prove that, given $\Sigma_H\in\mathscr F_\epsilon,$
the first and second eigenvalues $\lambda_1<\lambda_2$
of the Jacobi operator $L$ on $C^{\infty}(\Sigma_H)$
satisfy $\lambda_1<0=\lambda_2.$

Denote by $N$ the inward unit normal to $\Sigma_H.$
Since $\partial_t$ is a Killing field on $\qr,$ {it follows from
\eqref{eq-fundamental} that} the \emph{angle
function} $\theta:=\langle N,\partial_t\rangle$
belongs to the kernel of  $L.$ In addition, $\theta$ changes sign on $\Sigma_H,$ for it
takes both  values $-1$ and $1$ on the critical points of the height function
of $\Sigma_H.$ On the other hand,
it is a well known fact that any eigenfunction to the first eigenvalue
$\lambda_1$ of $L$ is either positive or negative. Hence, we must have  $\lambda_1<0$
(notice that this argument applies to any compact CMC hypersurface of $\qr$).

The proof that $\lambda_2=0$ is totally analogous to the one given
by Souam in~\cite[Theorem 2.2]{souam} in the case $n=2$, so we shall just sketch it here.
Let $\Sigma_H^-$ and $\Sigma_H^+$ be the two hemispheres of $\Sigma_H$ separated by $P_H$,
its horizontal hyperplane of symmetry. By using the angle function as in the preceding paragraph,
one concludes that $0$ is the first eigenvalue of $L$  for the Dirichlet problem
on each of these hemispheres.

Now, let $P\subset\mathbb Q_\epsilon^n$ be a totally geodesic hypersurface containing $o,$
and consider the vertical hyperplane $\Pi:=P\times\R$ of $\qr$, which is clearly a hyperplane
of symmetry of $\Sigma_{H}.$
Choose $q\in P$ such that $q$ is different from $o$ and, in the case $\epsilon=1,$ from its antipodal,
and take the Killing field $X_q$ associated to a one parameter family
of rotations around the axis $\{q\}\times\R.$ Then,
the function $\langle X_q,N\rangle$ belongs to $\ker L$ and vanishes precisely on $\Pi\cap\Sigma_{H}$,
which implies that $0$ is the first eigenvalue of $L$  for the Dirichlet problem on each
of the two connected components of $\Sigma_H-(\Sigma_H\cap\Pi)$.

Assume, by contradiction, that $\lambda_2<0$, and denote by
$E_{\lambda_2}$ the eigenspace associated to $\lambda_2.$ By means of
Courant's nodal theorem, and using the fact that $0$ is the first eigenvalue
of $L$  for the Dirichlet problems mentioned above,
one shows that the eigenfunctions
of $E_{\lambda_2}$ are invariant by reflection about $P_H$ and by rotations around the
axis $\{o\}\times\R$ of $\Sigma_H.$ From the latter property, the (compact) zero set
$u^{-1}(0)\subset\Sigma_H$ of
any such eigenfunction $u$ is invariant by rotations around $\{o\}\times\R$, and then
$u^{-1}(0)$ is a finite union of horizontal $(n-1)$-spheres. However, by Courant's nodal theorem,
$\Sigma-u^{-1}(0)$ has exactly two connected components, which implies that
$u^{-1}(0)$ consists of exactly one $(n-1)$-sphere. Since $u^{-1}(0)$ is also invariant
by reflection about $P_H,$ we conclude that it must coincide with $\Sigma_H\cap P_H$,
which contradicts that $0>\lambda_2$ is the first eigenvalue of $L$  for the Dirichlet problem
on the hemispheres $\Sigma_H^-$ and $\Sigma_H^+$.
Therefore, we have $\lambda_2=0,$
so that the condition {\bf A} in Koiso's Theorem is fulfilled.

Still following Souam, we fix $H^*>C_\epsilon$ and consider $\mathscr F_\epsilon$
as a one-parameter family $\Phi_H$ of CMC immersions  of $\Sigma_{H^*}$ into $\qr.$ Then, denoting
by $\xi$ the corresponding variational field on $\Sigma_{H^*},$
we have from  equality \eqref{eq-fundamental}
that the function $f:=\langle \xi,N\rangle$ satisfies  $Lf=1$
on $\Sigma_{H^*},$  since the parameter $t$ of this variation is $H$ itself. In particular,
for all $u\in\ker L,$ one has
\[
\int_{\Sigma_{H^*}}u=\int_{\Sigma_{H^*}}uLf=\int_{\Sigma_{H^*}}fLu=0,
\]
for $L$ is self-adjoint. Therefore, the condition in {\bf A}-(ii) of Koiso's Theorem
holds for the sphere $\Sigma_{H^*}.$

Now, set  $f=u+v,$ where $u\in\ker L$ and $v\in(\ker L)^\perp.$ Then,
$Lv=Lf=1$ and
\begin{equation} \label{eq-intf}
\int_{\Sigma_{H^*}}v=\int_{\Sigma_{H^*}}f.
\end{equation}

For $\epsilon=-1,$ we have from Theorem~\ref{th-nestedspheres} that the whole family
$\mathscr F_\epsilon$ is nested, which clearly implies that $f$ is nonnegative
and non identically zero on $\Sigma_{H^*}.$ Thus, considering \eqref{eq-intf},
we get from the final statement in {\bf A}-(ii) of Koiso's Theorem that all rotational CMC spheres
of $\h^n\times\R$ are stable. This proves part {\bf A} of our theorem.

Assume now that $\epsilon=1.$ Considering the area functional
$\mathcal A=\mathcal A(H)$ of the variation $\Phi_H$,
we have from~\eqref{eq-aprimeandvprime} and~\eqref{eq-intf} that
\[
\mathcal A'(H^*)=-\int_{\Sigma_{H^*}}H^*f=-H^*\int_{\Sigma_{H^*}}v.
\]
Hence, by assertion {\bf A}-(ii) of Koiso's Theorem, the rotational sphere
$\Sigma_{H^*}$ is stable if and only if $\mathcal A'(H^*)\le 0.$

It will be convenient to consider the correspondence~\eqref{eq-correspondence}
between the mean curvatures and
the radiuses of the spheres $\Sigma_H$ in $\mathscr F_\epsilon,$
and then write $\mathcal A$ as a function of the radius $s_0,$ that is,
\[
\mathcal A(s_0):={\rm Area}\,(\Sigma_{H(s_0)}), \,\,\, s_0\in (0,\pi).
\]
Then, since $H$ is a decreasing function of $s_0$, it follows from the  considerations
of the preceding paragraph that the following assertion holds.
\begin{claim} \label{claim-stability}
A rotational sphere $\Sigma_{H(s_0)}$ of $\mathscr F_\epsilon$ is stable if and only if  $\mathcal A'(s_0)\ge 0.$
\end{claim}

Taking Claim~\ref{claim-stability} into account, we shall
determine the constant $H_0$ of the statement of the theorem
by determining the maximum radius $s_0=s_0(H_0)$
such  that $\mathcal A'\ge 0$ in
$(0,s_0(H_0)].$ The constant $H_1>H_0,$ instead, will be determined
by considering Claim~\ref{claim-isoperimetric}, and by determining
the maximum radius $s_0=s_0(H_1)<s_0(H_0)$ for which
$\mathcal A$ is bounded above by $2\omega_n$ in $(0,s_0(H_1)]$
(see Fig.~\ref{fig-areagraph}).

\begin{figure}[htb]
\includegraphics[scale=.8]{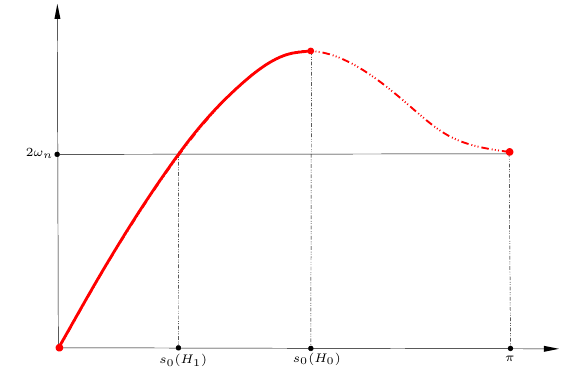}
\caption{\small An illustrative graph of the area function $\mathcal A=\mathcal A(s_0),$ considered in the proof of
Theorem~\ref{th-stabilityH-spheres}. By Claims~\ref{claim-isoperimetric} and~\ref{claim-stability}, and from
the growth of $\mathcal A,$ one has that all CMC spheres with radius in the interval $(0,s_0(H_0)]$ are stable, being
the ones with radius in $(0,s_0(H_1)]$ isoperimetric. The spheres with radius in $(s_0(H_1),s_0(H_0)]$
are all stable and non-isoperimetric.
For $n>2$, the growth behavior  of  $\mathcal A$ in the interval
$(s_0(H_0),\pi)$ is unknown. It is expected that, as in the case $n=2$,
$\mathcal A$ has only one critical point $s_0=s_0(H_0)$, in which case  it would be strictly decreasing in
$(s_0(H_0),\pi)$.}
\label{fig-areagraph}
\end{figure}

In order to  determine the constants $s_0(H_0)>s_0(H_1),$
we shall also consider the volume  $\mathcal V$ of the spherical regions
$\Omega_{H(s_0)}$ as a function of $s_0,$ that is,
$$\mathcal V(s_0):={\rm Vol}\,(\Omega_{H(s_0)}), \,\,\, s_0\in (0,\pi).$$
By Theorem~\ref{th-nestedspheres}, the spheres $\Sigma_{H(s_0)}$
converge to a double copy of the horizontal hyperplane
$\s^n\times\{0\}$ as $s_0\to\pi$. Therefore,
\[
\lim_{s_0\to\pi}\mathcal A(s_0)=2\omega_n \quad\text{and}\quad \lim_{s_0\to\pi}\mathcal V(s_0)=0.
\]

We have from the second of the above equalities  that $\mathcal V'(\bar s_0)<0$ for some
$\bar s_0>0$ sufficiently close to $\pi$.
In addition, since the spheres $\Sigma_{H(s_0)}$ are all CMC, we have (see, e.g.,~\cite[Lemma 5.2]{pedrosa})
\begin{equation} \label{eq-aprimevprimeagain}
\mathcal A'(s_0)=H(s_0)\mathcal V'(s_0) \,\,\, \forall s_0\in (0,\pi),
\end{equation}
which implies that $\mathcal A'(\bar s_0)<0.$ Hence, by Claim~\ref{claim-stability},
$\Sigma_{H(\bar s_0)}$ is unstable. In particular, $\mathcal A(\bar s_0)>2\omega_n.$
Otherwise, by Claim~\ref{claim-isoperimetric}, $\Sigma_{H(\bar s_0)}$ would be isoperimetric, and so it would
be stable. On the other hand, by Theorem~\ref{th-uniquenessPI}, the spheres
$\Sigma_{H(s_0)}$ are  isoperimetric for any sufficiently small $s_0>0$, and so they
are all stable. Hence, from Claims~\ref{claim-isoperimetric} and~\ref{claim-stability},  $\mathcal A(s_0)\le2\omega_n$
and $\mathcal A'(s_0)\ge 0$ for any  such  $s_0.$ Notice that, in this setting, $\mathcal A$ must increase from zero
to a value beyond $2\omega_n,$ since $\mathcal A(\bar s_0)>2\omega_n$ and, as we argued above,
there is no $s_0>0$ such that $\mathcal A'(s_0)<0$ and
$\mathcal A(s_0)\le2\omega_n$ (see Fig.~\ref{fig-areagraph}).

It follows from the above considerations that there exists a maximum $\tilde s_0\in(0,\pi)$ such that
$\mathcal A'\ge 0$ in $(0,\tilde s_0],$ and that $\tilde s_0$ satisfies $\mathcal A(\tilde s_0)>2\omega_n.$
Also, there exists a maximum $\hat s_0\in(0,\tilde s_0)$ such that $\mathcal A\le 2\omega_n$ in
$(0,\hat s_0].$ Defining $H_0<H_1$ by the equalities $s_0(H_0)=\tilde s_0$ and
$s_0(H_1)=\hat s_0,$ it follows  easily from Claims~\ref{claim-isoperimetric} and~\ref{claim-stability}
that $H_0$ and $H_1$ have the asserted properties.
This concludes our proof.
\end{proof}

\uniqueness*
\begin{proof}
Recalling~\cite[Theorem 2]{pedrosa}, any solution to the isoperimetric problem
in \,$\s^n\times\R$ is either a spherical region $\Omega_H$ or a cylindrical section $\Omega_{[a,b]}$.
With the notation of the proof of Theorem~\ref{th-stabilityH-spheres}-{\bf B}, we set
\[
\overbar{\mathcal V}:=\mathcal V(s_0(H_1))={\rm Vol}(\Omega_{H_1}).
\]

As we have seen, in the interval $I:=(0,s_0(H_1)),$
the area function $\mathcal A=\mathcal A(s_0)$ is bounded above by
$2\omega_n=\mathcal A(s_0(H_1))$ and satisfies
$\mathcal A'\ge 0$, so that $\mathcal A$ is increasing in $I$.
In particular, a point $s_0\in I$ satisfies
$\mathcal A(s_0)=2\omega_n$ if and only if
$\mathcal A$ is constant on $[s_0,s_0(H_1)]$.

By~\eqref{eq-aprimevprimeagain},
$\mathcal V'\ge 0$ on $I$.  Therefore, if
$\mathcal V_0<\overbar{\mathcal V}=\mathcal V(s_0(H_1)),$
there exists   $s_0\in I$ such that
$\mathcal V(s_0)=\mathcal V_0.$ For such an $s_0,$  $\mathcal A(s_0)\le 2\omega_n.$
In fact, $\mathcal A(s_0)<2\omega_n.$ If not, by the last statement of the
preceding paragraph, $\mathcal A$ would be constant on $[s_0,s_0(H_1)]$.
Again by~\eqref{eq-aprimevprimeagain}, the same would be true for $\mathcal V,$
which would contradict that $\mathcal V(s_0)=\mathcal V_0<\overbar{\mathcal V}$.
Therefore, from Claim~\ref{claim-isoperimetric},
no cylindrical section of $\s^n\times\R$ is a solution to the isoperimetric
problem for prescribed volume $\mathcal V_0$. This proves (i).

To prove (ii) and (iii), first notice that $\mathcal A$ cannot decrease below $2\omega_n$
in $(s_0(H_1),\pi).$ Otherwise, as we argued in the proof of Theorem~\ref{th-stabilityH-spheres},
there would exist an unstable isoperimetric CMC sphere in $\s^n\times\R.$ Then, we have
\begin{claim} \label{claim-last}
$\mathcal A(s_0)<2\omega_n$ if and only if $s_0\in I:=(0,s_0(H_1))$.
\end{claim}

Assume now that $\mathcal V_0>\overbar{\mathcal V}$.
Then, since $\mathcal V$ is increasing in $(0,s_0(H_1)],$ one has
\begin{equation} \label{eq-inequalitiesvzerov1}
\mathcal V(s_0)\le\mathcal V(s_0(H_1))=\overbar{\mathcal V}<\mathcal V_0\,\,\, \forall s_0\in (0,s_0(H_1)].
\end{equation}
Thus, the radius of a  spherical region $\Omega_{H(s_0)}$ of volume $\mathcal V_0,$ if any, is
in the interval $(s_0(H_1),\pi).$ In this case, it follows from Claim~\ref{claim-last}
that $\mathcal A(s_0)\ge 2\omega_n,$ which, together with Claim~\ref{claim-isoperimetric},
implies  that any cylindrical section $\Omega_{[a,b]}$ with $b-a=\mathcal V_0/\omega_n$ is
a solution to the isoperimetric problem for prescribed volume $\mathcal V_0.$
This proves (ii) and also (iii), since $\mathcal A(s_0(H_1))=2\omega_n.$

The first part of the final statement  follows from Claim~\ref{claim-isoperimetric} and
assertions (i)--(iii). For the second part,
we observe that, since  $\mathcal V$ is increasing an nonconstant in $I$,
given $\delta\in(0,\overbar{\mathcal V}),$
there exists $H_\delta>H_1$  such that $\mathcal V(s_0(H_\delta))=\overbar{\mathcal V}-\delta.$
Also, it is clear that $\mathcal V$ cannot be constant
in the interval $I_\delta:=(s_0(H_\delta),s_0(H_1))$,
since $\mathcal V(s_0(H_\delta))=\overbar{\mathcal V}-\delta<\mathcal V(s_0(H_1))$.
Hence, there exists $s_0\in I_\delta$ satisfying $\mathcal V'(s_0)>0,$
so that $\mathcal V$ is strictly increasing
in a neighborhood of $s_0.$ In particular, $\Omega_{H(s_0)}$
is the only solution to the isoperimetric
problem for prescribed volume
$\mathcal V_0:={\rm Vol}(\Omega_{H(s_0)})$.
This finishes the proof.
\end{proof}

\section{Acknowledgements}
We are indebted to R. Souam for many useful conversations and valuable suggestions,
specially the ones regarding the matter discussed in Remark~\ref{rem-flawed}.
We also thank  the anonymous referee
for the careful reading and all  useful comments and suggestions.
This study was partially supported by Coordenação de Aperfeiçoamento de Pessoal
de Nível Superior - Brasil (CAPES) - Finance Code 001,
INdAM-GNSAGA and PRIN-2022AP8HZ9.

\end{document}